\documentclass[12pt]{article}

\usepackage{amstext    }
\usepackage{amsthm    }
\usepackage{a4}
\usepackage[mathscr]{eucal}
\usepackage{mathrsfs}
\usepackage{indentfirst}
\usepackage{amsmath,bm}

\usepackage{amsmath}
\usepackage{amssymb}
\usepackage{amscd}

\newtheorem{theorem}{Theorem}[section]
\newtheorem{definition}[theorem]{Definition}
\newtheorem{proposition}[theorem]{Proposition}
\newtheorem{corollary}[theorem]{Corollary}
\newtheorem{lemma}[theorem]{Lemma}
\newtheorem{remark}[theorem]{Remark}

\newtheorem{example}[theorem]{Example}

\newcommand{\supp}{{\rm supp}}

\newcommand{\dist}{{\rm dist}}

\title{Equidistribution of Zeros of Random Holomorphic Sections for Moderate Measures}

\author{Guokuan SHAO}

\begin{document}

\maketitle

\begin{abstract}
We establish an equidistribution theorem for the zeros of random holomorphic sections of high powers of a positive
holomorphic line bundle. The equidistribution is associated with a family of singular moderate measures.
We also give a convergence speed for the equidistribution.
\end{abstract}

\noindent
{\bf Classification AMS 2010:} 32A60, 32L10, 32U40.


\section{Introduction}
Distribution of zeros of random polynomials is a classical subject.
Waring \cite{t} used a probabilistic method to determine the number of imaginary zeros of an algebraic polynomial.
More rigorous and systematic research started with the paper of Bloch-P\'{o}lya \cite{bp} in 1930s.
They gave an order of the expected number of real roots of certain random algebraic polynomial equation.
Littlewood-Offord, Hammersley, Kac and Erd\"{o}s-Tur\'{a}n developed this field motivated by Bloch-P\'{o}lya's work.
See \cite{bd, bs, sv} for a review and complete references.

We introduce the basic setting in the paper.
Let $\omega_{FS}$ be the standard K\"{a}hler form induced by the Fubini-Study metric on $\mathbb{P}^{k}$
normalized by $\int_{\mathbb{P}^{k}}\omega_{FS}^{k}=1$. Let $X$ be a projective manifold of dimension $k$, $L$ an ample line bundle over $X$.
Fix a Hermitian metric $h$ on $L$ such that the curvature form $\omega$ is K\"{a}hler on $X$.
Then $\omega$ represents the first Chern class $c_{1}(L)$
with $\int_{X}\omega^{k}=c_{1}(L)^{k}\in\mathbb{Z}^{+}$. Let $L^{n}$ be the $n$th tensor product of $L$.
Denote by $H^{0}(X,L^{n})$ the space of all holomorphic sections of $L^{n}$.
Let $\mathbb{P}H^{0}(X,L^{n})$ be the associated projective space. We denote by $\omega_{FS}$
its normalized Fubini-Study form. Set $k_{n}:=\dim\mathbb{P}H^{0}(X,L^{n})$,
where $k_{n}$ is given by the Hilbert polynomial whose dominant term is $c_{1}(L)^{k}n^{k}/k!$ \cite{ks}. Let $s_{n}\in\mathbb{P}H^{0}(X,L^{n})$.
Denote by $[\bf{Z}_{s_{n}}]$ the current defined by the zero set of $s_{n}$. Set $\mathbb{P}^{X}=\prod_{n\geq 1}\mathbb{P}H^{0}(X,L^{n})$.

In order to state our theorem, we also need the following terminologies.
Fix some exponent $0<\rho<1$,
a function $u: M\to\mathbb{R}$ defined on a metric space $(M, \rm dist)$
is said to be of class $\mathscr{C}^{\rho}$ with modulus $c$ if
\begin{equation*}
\sup_{\substack{x,y\in M \\x\neq y}}\frac{|u(x)-u(y)|}{\dist(x,y)^{\rho}}\leq c.
\end{equation*}
Consider a complex manifold $M$ with a fixed volume form, let $\gamma$ be a closed real current of bidegree $(1,1)$ on $M$.
An upper-semi continuous function $u: M\to [-\infty, \infty)$ in $L^{1}_{loc}(M)$ is said to be $\gamma$-p.s.h. if $dd^{c}u+\gamma\geq 0$.

Let $\sigma_{n}$ be the probability Lebesgue measure on $\mathbb{P}H^{0}(X,L^{n})$ and
$\sigma$ the product measure of these ones on $\mathbb{P}^{X}$.
Shiffman-Zelditch \cite{sz} proved that the sequence of currents $\{\frac{1}{n}[\bf{Z}_{s_{n}}]\}$ converges weakly to $\omega$ for $\sigma$-almost everywhere
$(s_{n})\in\mathbb{P}^{X}$. They used the potential-theoretic approach from Forn\ae ss-Sibony's work \cite{fs}.
Dinh-Sibony \cite{ds1} generalized the result and obtained a good estimate of the convergence speed
over a projective manifold endowed with a smooth positively-curved metric.
Consequently, they constructed a singular measure with real coefficients which satisfies equidistribution property (cf. \cite[Corollary 7.4]{ds1}).
Dinh-Ma-Marinescu \cite{dmm} recently established the equidistribution for a semipositive singular Hermitian line bundle.
When the Lebesgue measures in Shiffman-Zelditch's result are replaced by moderate measures with H\"{o}lder potentials (see Sections 2, 3),
we have our main theorem as follows which gives a concrete large family of singular moderate measures that satisfies equidistribution property.
It can be regarded as a perturbation of standard measures induced by Fubini-Study metric.
\begin{theorem}
Let $L$ be an ample line bundle over a projective manifold $X$ of dimension $k$ and $0<\rho<1$ an exponent.
Then there exists a constant $c=c(X,L,\rho)>1$ with the following property.
For each $n\geq 1$, $1\leq j\leq k_{n}$, let $u_{n,j}:\mathbb{P}H^{0}(X,L^{n})\rightarrow\mathbb{R}$ be a function and $\xi_{n},\epsilon_{n}>0$ two numbers such that£º

\noindent (i)\quad $u_{n,j}$ is of class $\mathscr{C}^{\rho}$ with modulus $\xi_{n}$, $\forall 1\leq j\leq k_{n}$;

\noindent (ii)\quad $u_{n,j}$ is $\epsilon_{n}\omega_{FS}$-p.s.h., $\forall 1\leq j\leq k_{n}$;

\noindent (iii)\quad $\xi_{n}\leq 1/c^{n^{k}}, \epsilon_{n}\leq 1/c^{n^{k}}$.

\noindent Let $\sigma_{n}=(dd^{c}u_{n,1}+\omega_{FS})\wedge(dd^{c}u_{n,2}+\omega_{FS})\wedge\cdot\cdot\cdot\wedge(dd^{c}u_{n,k_{n}}+\omega_{FS})$ be the probability measure on $\mathbb{P}H^{0}(X,L^{n})$.
Endow $\mathbb{P}^{X}$ with the product measure $\sigma=\prod_{n\geq 1}\sigma_{n}$.
Then for almost everywhere $s=(s_{n})\in\mathbb{P}^{X}$ with respect to $\sigma$,
the sequence of currents $\{\frac{1}{n}[\bf{Z}_{s_{n}}]\}$ converges weakly to $\omega$.
\end{theorem}
The following result gives a convergence speed for the equidistribution in Theorem 1.1.
\begin{theorem}
In the setting of Theorem 1.1, there exist subsets $E_{n}\subset\mathbb{P}H^{0}(X,L^{n})$ and a positive constant
$C$ depending only on $X, L$ such that for all $n$ sufficiently large, we have
\begin{equation*}
\sigma_{n}(E_{n})\leq\frac{C}{n^{2}} \quad \text{and} \quad
|\bigl<\frac{1}{n}[\bm{Z}_{\bm{s_{n}}}]-\omega, \psi\bigr>|\leq\frac{C\log n}{n}\|\psi\|_{\mathscr{C}^{2}},
\end{equation*}
for any point $s_{n}\in\mathbb{P}H^{0}(X,L^{n})\setminus E_{n}$ and any $(k-1,k-1)$-form $\psi$ of class $\mathscr{C}^{2}$.
\end{theorem}

The paper is organized as follows. In Section 2 we recall the notion of moderate measure
and give an estimate for moderate measures with H\"{o}lder continuous potential on $\mathbb{P}^{k}$.
In Section 3 we introduce some notions and theorems for the equidistribution of zeros. We then apply the results in Section 2
to prove the main theorem and give a very explicit example. We conclude Section 3 with the proof of Theorem 1.2.

\section{Estimate for moderate measures}
In this section, we give an estimate for moderate measures on $\mathbb{P}^{k}$.
Some preliminaries for definitions and properties are needed.

Let $(X, \omega)$ be a compact K\"{a}hler manifold of dimension $k$, $\omega^{k}$ its standard volume form.
We say that a function $\phi$ on $X$ is quasiplurisubharmonic (q.p.s.h.) if it is $c\omega$-p.s.h. for some constant $c>0$.
In fact, it is locally the difference of a p.s.h. function and a smooth one.
Consider a positive measure $\mu$ on $X$, $\mu$ is said to be PLB if all the q.p.s.h. functions are $\mu$-integrable.
When $\dim X=1$, $\mu$ is PLB if and only if it admits a local bounded potential \cite{ds2}.
Let
\begin{equation}
\mathcal{F}=\{\phi ~q.p.s.h. ~on~ X: dd^{c}\phi\geq -\omega, \max_{X}\phi =0\}.
\end{equation}
$\mathcal{F}$ is compact in $L^{p}(X)$ and bounded in $L^{1}(\mu)$ when $\mu$ is a PLB measure, see \cite{ds1}.
\begin{definition}
Let $\mu$ be a {\rm PLB} measure on $X$. We say that $\mu$ is $(c,\alpha)$-moderate for some constants $c>0, \alpha >0$ if
\begin{equation*}
\int_{X}\exp(-\alpha \phi)d\mu\leq c
\end{equation*}
for all $\phi\in\mathcal{F}$. The measure $\mu$ is called moderate if there exist constants $c>0, \alpha >0$ such that it is $(c,\alpha)$-moderate.
\end{definition}
For example, $\omega^{k}$ is moderate \cite{hl2}.
When $X=\mathbb{P}^{k}$, we recall the following proposition \cite[Corollary A.5]{ds1},
\begin{proposition}
There are constants $c_{0}>0$ and $\alpha_{0} >0$ independent of $k$ such that
\begin{equation*}
\int_{\mathbb{P}^{k}}\exp(-\alpha_{0}\phi)\omega^{k}_{FS}\leq c_{0}k, \quad \forall \phi\in\mathcal{F}.
\end{equation*}
\end{proposition}
\begin{remark}
We have a general definition for locally moderate measure on a complex manifold $X$ of dimension $k$.
The measure $\mu$ is locally moderate if for any open set $U\subset X$, any compact subset $K\subset U$ and any compact family $\mathcal{G}$ of q.p.s.h. functions on $U$,
there are constants $\alpha >0, c>0$ such that
\begin{equation*}
\int_{K}\exp(-\alpha\phi)d\mu\leq c, \quad \forall\phi\in\mathcal{G}.
\end{equation*}
As a consequence, $\mathcal{G}$ is bounded in $L^{p}_{loc}(\mu)$, where $p\geq 1$.
\end{remark}
The following lemma gives an alternative definition of moderate measures \cite{dn}.
\begin{lemma}
A PLB measure $\mu$ is moderate if and only if there exist two constants $c'>0, \alpha'>0$ such that
\begin{equation*}
\mu\{z\in K: \phi(z)< -M\}\leq c'e^{-\alpha'M}
\end{equation*}
for any $M\geq 0$ and $\phi\in\mathcal{F}$.
\end{lemma}
\begin{remark}
We can take $c'=c$, $\alpha'=\alpha$ when $c,\alpha$ are given
and take $c=2c'$, $\alpha =\alpha'/2$ when $c', \alpha'$ are given.
\end{remark}
Let $S$ be a positive closed current of bidegree $(p,p)$ on $X$, the trace measure is $\sigma_{S}=S\wedge\omega^{k-p}$
for a fixed Hermitian form $\omega$ on $X$. Here $X$ may not be compact.
$S$ is said to be locally moderate if its trace measure is locally moderate.
If $u$ is a continuous real-valued function and $uS$ defines a current on $X$ (for example, if supp $u\subset$ supp$S$),
then $dd^{c}(uS)$ is well defined.
We say that $u$ is $S$-p.s.h. if $dd^{c}(uS)$ is a positive current.
Dinh-Nguy\^{e}n-Sibony \cite[Theorem 1.1]{dns1} proved the following theorem. We improve their method quantitatively in this section.
\begin{theorem}
Let $S$ be a locally moderate positive closed $(p,p)$-current on a complex manifold $X$. If $u$ is a H\"{o}lder continuous $S$-p.s.h. function,
then $dd^{c}(uS)$ is locally moderate.
\end{theorem}
\begin{corollary}
Let $u$ be a H\"{o}lder continuous p.s.h. function on $X$. Then the Monge-Amp$\grave e$re currents $(dd^{c}u)^{p}$ are locally moderate.
\end{corollary}

Denote by $S^{k}$ the unit sphere on $\mathbb{R}^{k+1}$, $B_{1}$ the unit ball in $\mathbb{C}^{k}$.
Let $\pi: S^{2k+1}\rightarrow\mathbb{P}^{k}$ be the natural projection map.
More precisely, set $z_{j}=x_{j}+iy_{j}, x_{j},y_{j}\in\mathbb{R}, 0\leq j\leq k$, when $\sum_{j=0}^{k}|z_{j}|^{2}=1$, we have
$\pi(x_{0},y_{0},...,x_{k},y_{k})=[z_{0},...,z_{k}]$. Let $U_{0}=\{[z_{0},...,z_{k}]\in\mathbb{P}^{k},z_{0}\not =0\}$. There is a natural isomorphism
\begin{equation}
\theta:U_{0}\rightarrow\mathbb{C}^{k} ,
[z_{0},...,z_{k}]\rightarrow(z_{1}/z_{0},...,z_{k}/z_{0})
\end{equation}
Let $K_{0}=\theta^{-1}(B_{1})$.
$K_{0}$ is a neighbourhood of $[1,0,...,0]$ in $\mathbb{P}^{k}$.
$\pi^{-1}(K_{0})=\{(x_{0},y_{0},...,x_{k},y_{k})\in S^{2k+1}, \sum_{j=1}^{k}|z_{j}|^{2}\leq |z_{0}|^{2}\}$.
Let $S_{0}=\{(x_{0},y_{0},...,x_{k},y_{k})\in S^{2k+1}, x_{0}>\frac{1}{\sqrt{2}}\}$. It's obvious that $S_{0}\subset\pi^{-1}(K_{0})$
and $\pi(S_{0})$ is a neighbourhood of $[1,0,...,0]$. By the homogeneity of $S^{2k+1}$(resp. $\mathbb{P}^{k}$),
there is a neighbourhood $S_{0}^{\prime}$ (resp. $\pi(S_{0}^{\prime})$) of any point $(x_{0},y_{0},...,x_{k},y_{k})$
(resp. $[z_{0},...,z_{k}]$) which is the image of $S_{0}$ (resp. $\pi(S_{0})$) by rotations (resp. unitary transformations).
We say that $S_{0}^{\prime}$ (resp. $\pi(S_{0}^{\prime})$) is similar to $S_{0}$ (resp. $\pi(S_{0})$).
Since $\mathbb{P}^{k}$ is compact,
there are finitely many such neighbourhoods $\pi(S_{0})$ that cover $\mathbb{P}^{k}$.
Denote by $M_{k}$ the minimum number of such neighbourhoods $\pi(S_{0})$ that cover $\mathbb{P}^{k}$.
We have the following lemma.
\begin{lemma}
Let $K_{0}$ be as above. For any point $z\in\mathbb{P}^{k}$, there exists a neighbourhood $K_{z}$ of $z$ which is similar to $K_{0}$.
Denote by $N_{k}$ the minimum number of such neighbourhoods $K_{0}$ that cover $\mathbb{P}^{k}$. Then $N_{k}=O(8^{k})$.
\end{lemma}
\begin{proof}
Since $\pi(S_{0})\subset K_{0}$, then $M_{k}\geq N_{k}$. So it remains to prove that $M_{k}=O(8^{k})$.
We endow $S^{2k+1}$ with the great-circle distance. $S_{0}$ can be regarded as an open ball with central point $[1,0,...,0]$ of radius $\frac{\pi}{4}$.
Denote $S_{0}$ by $B([1,0,...,0],\frac{\pi}{4})$.
Let $S_{1}=B([1,0,...,0],\frac{\pi}{8})=\{(x_{0},y_{0},...,x_{k},y_{k})\in S^{2k+1}, x_{0}>\frac{\sqrt{2+\sqrt{2}}}{2}\}$.
We first consider the open balls of radius $\frac{\pi}{8}$. All of them are similar to each other.
We put the maximal number of balls $B(z_{1},\frac{\pi}{8}),...,B(z_{m_{k}},\frac{\pi}{8})$ in $S^{2k+1}$
such that all of them are disjoint mutually. Then $S^{2k+1}=\bigcup_{j=1}^{j=m_{k}}B(z_{j},\frac{\pi}{4})$.
If there exists a point $w\in S^{2k+1}\setminus\bigcup_{j=1}^{j=m_{k}}B(z_{j},\frac{\pi}{4})$,
then the great-circle distance between $w$ and $z_{j}$ is larger than or equal to $\frac{\pi}{4}$ for all $1\leq j\leq m_{k}$.
Hence $B(w,\frac{\pi}{8})\subset S^{2k+1}\setminus\bigcup_{j=1}^{j=m_{k}}B(z_{j},\frac{\pi}{8})$, contradicts with the maximality.
Then $M_{k}\leq m_{k}\leq Vol(S^{2k+1})/Vol(S_{1})$, the last inequality is due to the mutual disjointedness.
It means that $N_{k}=O(Vol(S^{2k+1})/Vol(S_{1}))$.

We now use the spherical coordinate for $S^{2k+1}$.
Let $x_{0}=\cos\theta_{1}, y_{0}=\sin\theta_{1}\cos\theta_{2}$, $..., x_{k}=\sin\theta_{1}\sin\theta_{2}\cdots \sin\theta_{2k}\cos\theta_{2k+1},
y_{k}=\sin\theta_{1}\sin\theta_{2}\cdots \sin\theta_{2k}\sin\theta_{2k+1}$.
Then the volume element of $S^{2k+1}$ is $d_{S^{2k+1}}V=\sin^{2k}\theta_{1}\sin^{2k-1}\theta_{2}\cdots \sin\theta_{2k}d\theta_{1}d\theta_{2}\cdots d\theta_{2k+1}$.
\begin{equation*}
Vol(S^{2k+1})=\int_{0}^{\pi}\sin^{2k}\theta_{1}\,d\theta_{1} \int_{0}^{\pi}\sin^{2k-1}\theta_{2}\,d\theta_{2}\cdots
\int_{0}^{\pi}\sin\theta_{2k}\,d\theta_{2k} \int_{0}^{2\pi}\,d\theta_{2k+1}
\end{equation*}
\begin{equation*}
Vol(S_{1})=\int_{0}^{\frac{\pi}{8}}\sin^{2k}\theta_{1}\,d\theta_{1} \int_{0}^{\pi}\sin^{2k-1}\theta_{2}\,d\theta_{2}\cdots
\int_{0}^{\pi}\sin\theta_{2k}\,d\theta_{2k} \int_{0}^{2\pi}\,d\theta_{2k+1} .
\end{equation*}
This yields $O(Vol(S^{2k+1})/Vol(S_{1}))=O(\int_{0}^{\pi}\sin^{2k}\theta_{1}\,d\theta_{1}/\int_{0}^{\frac{\pi}{8}}\sin^{2k}\theta_{1}\,d\theta_{1})$.

Then it suffices to show that
$\int_{0}^{\pi}\sin^{2k}\theta_{1}\,d\theta_{1}/\int_{0}^{\frac{\pi}{8}}\sin^{2k}\theta_{1}\,d\theta_{1}\leq 8^{k+1},\forall k\geq 7$.
When $k=7$, the inequality is right.
By induction on $k$ and the following integrals
\begin{equation*}
\int\sin^{2k}\theta_{1}\,d\theta_{1}=-\frac{\sin^{2k-1}\theta_{1}\cos\theta_{1}}{2k}+\frac{2k-1}{2k}\int\sin^{2k-2}\theta_{1}\,d\theta_{1},
\end{equation*}
the proof is reduced to show that $\int_{0}^{\frac{\pi}{8}}\sin^{2k}\theta_{1}\,d\theta_{1}\geq\frac{8}{7}\frac{1}{2k+1}\frac{\sqrt{2}}{4}(\frac{2-\sqrt{2}}{4})^{k}$.
By the relation between $\int_{0}^{\frac{\pi}{8}}\sin^{2k}\theta_{1}\,d\theta_{1}$ and $\int_{0}^{\frac{\pi}{8}}\sin^{2k+6}\theta_{1}\,d\theta_{1}$,
we have
\begin{equation*}
\begin{split}
&\int_{0}^{\frac{\pi}{8}}\sin^{2k}\theta_{1}\,d\theta_{1} \\
&\geq\frac{1}{2k+1}
\frac{\sqrt{2}}{4}(\frac{2-\sqrt{2}}{4})^{k}(1+\frac{2-\sqrt{2}}{4}\frac{2k+2}{2k+3}+(\frac{2-\sqrt{2}}{4})^{2}\frac{(2k+2)(2k+4)}{(2k+3)(2k+5)}) \\
\end{split}
\end{equation*}
Then the proof is completed.
\end{proof}

The following lemma is needed \cite[Lemma 2.3]{dns1}.
\begin{lemma}
Let $T$ be a positive closed current of bidegree $(k-1,k-1)$ and $u$ a $T$-p.s.h. function on a neighbourhood $U$
of the unit ball $B_{1}$ in $\mathbb{C}^{k}$. Suppose that $u$ is smooth on $B_{1-r}\setminus B_{1-4r}$ for a fixed number $0<r<1/4$.
If $\phi$ is a q.p.s.h. function on $U$,
$\chi$ is a smooth function with compact support on $B_{1-r}, 0\leq\chi_{k}\leq 1$ and $\chi_{k}\equiv 1$ on $B_{1-2r}$.
Then
\begin{equation*}
\begin{split}
&\int_{B_{1}}\chi\phi dd^{c}(uT)=-\int_{B_{1-r}\setminus B_{1-3r}}dd^{c}\chi\wedge\phi uT \\
&-\int_{B_{1-r}\setminus B_{1-3r}}d\chi\wedge\phi d^{c}u\wedge T +\int_{B_{1-r}\setminus B_{1-3r}}d^{c}\chi\wedge\phi du\wedge T \\
&+\int_{B_{1-r}}\chi udd^{c}\phi\wedge T.\\
\end{split}
\end{equation*}
\end{lemma}

Let $\mathcal{F}$ be defined in (1) when $X=\mathbb{P}^{k}$ and $\theta$ defined in (2).
The following lemma is crucial for the main proposition in this section.
\begin{lemma}
Let $u$ be of class $\mathscr{C}^{\rho}$ with modulus $\epsilon$ on a neighbourhood $U$ of $B_{1}$ in $\mathbb{C}^{k}$
with $dd^{c}u\geq 0$ in the sense of currents,
$0<\rho<1$. Set $\omega:=\frac{1}{2}dd^{c}\log(1+\|z\|^{2})$. Let $\mathcal{F}_{0}=\{\phi\circ\theta^{-1}$ on $U: \phi\in\mathcal{F}\}$
and $T$ a positive closed $(k-1,k-1)$-current.
If $T\wedge\omega$ is $(c,\alpha)$-moderate on $U$, then
\begin{equation*}
\int_{B_{1}}\exp(-\frac{\alpha\rho}{4}\phi)dd^{c}(uT)
\leq ck\epsilon (c_{1}e^{\alpha}+\frac{c_{2}}{\alpha})
\end{equation*}
where $c_{1},c_{2}$ are positive constants independent of $k$, $\rho$ and $T$.
\end{lemma}
\begin{proof}
We modify the function $u$ on $U$.
Subtracting a constant, we assume that $u\leq -\epsilon/2$ on $B_{1}$.
Consider the function $v(z)=\max(u(z), \epsilon A\log|z|)$ for a constant $A>0$ large enough such that
$v$ coincides with $u$ near the origin and $v(z)=\epsilon A\log|z|$ near the boundary of $B_{1}$.
For example, $A=\frac{1}{2}\log\frac{1}{1-4r}$.
$A$ is independent of the choice of $u$.
Fix $0<r<1/16$, we are allowed to assume that $u=\epsilon A\log|z|$ on $B_{1}\setminus B_{1-4r}$.
For the smooth function $\chi$ defined in Lemma 2.9,
we can assume that $\|\chi\|_{\mathscr{C}^{2}}<c_{3}$ for some constant $c_{3}>1$ large enough independent of $k$,
since the terms in the definition of the norm $\|\bullet\|_{\mathscr{C}^{2}}$ are smooth on the compact subset $\bar B_{1-r}\setminus B_{1-2r}$.
Set $\sigma_{T}=T\wedge\omega, \sigma_{T^{'}}=dd^{c}(uT), \phi_{M}=\max(\phi ,-M), \psi_{M}=\phi_{M-1}-\phi_{M}$, for $\phi\in\mathcal{F}_{0}, M\geq 0$. To prove the lemma, we need to estimate the mass of $dd^{c}(uT)$ on $\{\phi<-M\}$.
Since $\supp\chi\subset B_{1-r}$, hence
\begin{equation*}
\sigma_{T^{'}}\{\phi <-M\}\leq\int\chi\psi_{M}dd^{c}(uT).
\end{equation*}
Since $T$ is $(c,\alpha)$-moderate, then
\begin{equation*}
\sigma_{T}\{z\in B_{1-r}, \phi(z)\leq -M+1\}\leq ce^{\alpha}e^{-\alpha M}.
\end{equation*}
By Lemma 2.9, we have
\begin{equation}
\begin{split}
&\int_{B_{1}}\chi\psi_{M}dd^{c}(uT)=-\int_{B_{1-r}\setminus B_{1-3r}}dd^{c}\chi\wedge\psi_{M}uT \\
&-\int_{B_{1-r}\setminus B_{1-3r}} d\chi\wedge\psi_{M}d^{c}u\wedge T+\int_{B_{1-r}\setminus B_{1-3r}}d^{c}\chi\wedge\psi_{M}du\wedge T \\
&+\int_{B_{1-r}}\chi udd^{c}\psi_{M}\wedge T \\
\end{split}
\end{equation}
We know that
$\omega=\frac{1}{2}dd^{c}\log (1+\|z\|^{2})
=i\sum_{j,l=1}^{k}(\frac{dz_{j}\wedge d\bar z_{l}}{1+\|z\|^{2}}-\frac{\bar z_{j}z_{l}dz_{j}\wedge d\bar z_{l}}{(1+\|z\|^{2})^{2}}).$
By simple computations, the eigenvalues of the corresponding Hermitian matrix of $\omega$
are $\frac{1}{(1+\|z\|^{2})^{2}}$ and $\frac{1}{1+\|z\|^{2}}$ ($k-1$ times).
On the other hand, the eigenvalues of the corresponding Hermitian matrix of $i\sum_{j,l=1}^{k}dz_{j}\wedge d\bar z_{l}$ are $k$ and $0$ ($k-1$ times).
So there exists a constant $m_{1}>0$ small enough such that $\omega -\frac{m_{1}}{k}i\sum_{j,l=1}^{k}dz_{j}\wedge d\bar z_{l}>0$ on $B_{1}$.
Hence $|dd^{c}\chi\wedge uT|\leq |uc_{3}i\sum_{j,l=1}^{k}dz_{j}\wedge d\bar z_{l}\wedge T|\leq\epsilon A|\log (1-3r)|c_{3}\frac{k}{m_{1}}\sigma_{T}$.
Observing that $0\leq\psi_{M}\leq 1$, supp$\psi_{M}\subset\{\phi <-M+1\}$,
we obtain
\begin{equation*}
\left |\int_{B_{1-r}\setminus B_{1-3r}}dd^{c}\chi\wedge\psi_{M}uT\right |\leq\epsilon A|\log (1-3r)|c_{3}\frac{k}{m_{1}}ce^{\alpha}e^{-\alpha M}.
\end{equation*}
Since we know $u$ explicitly on supp($d\chi$), we obtain
\begin{equation*}
\left|\int_{B_{1-r}\setminus B_{1-3r}}d\chi\wedge\psi_{M}d^{c}u\wedge T\right|\leq\frac{\epsilon A}{1-3r}c_{3}km_{2}ce^{\alpha}e^{-\alpha M},
\end{equation*}
\begin{equation*}
\left|\int_{B_{1-r}\setminus B_{1-3r}}d^{c}\chi\wedge\psi_{M}du\wedge T\right|\leq\frac{\epsilon A}{1-3r}c_{3}km_{2}ce^{\alpha}e^{-\alpha M}.
\end{equation*}
for a constant $m_{2}>0$ large enough independent of $k$.
The sum of the first three terms is less than
\begin{equation}
c_{4}\epsilon kce^{\alpha}e^{-\alpha M}
\end{equation}
where $c_{4}=Ac_{3}(\frac{|\log (1-3r)|}{m_{1}}+\frac{2m_{2}}{1-3r})$ is independent of $k$ and $\rho$.

For the last integral in (3), we use a regularization procedure and the condition of
$\rho$-H\"{o}lder continuity of $u$.
Let $\{u_{\delta}\}$ be the smooth approximation of $u$ obtained by convolution.
For some fixed $0<\delta<1$ small enough, $u_{\delta}$ is defined in a neighborhood of $\bar B_{1-r}$.
There exists a suitable function $u_{\delta}$ satisfying that
$\|u_{\delta}\|_{\mathscr{C}^{2}}\leq\epsilon\delta^{-(2-\rho)}$ and
$\|u-u_{\delta}\|_{\infty}\leq\epsilon\delta^{\rho}$,
where the latter inequality follows from that $u$ is of class $\mathscr{C}^{\rho}$ with modulus $\epsilon$.
The above two inequalities are independent of $k$.
We write
\begin{equation*}
\begin{split}
&\int_{B_{1}}\chi udd^{c}\psi_{M}\wedge T \\
&=\int\chi dd^{c}\psi_{M}\wedge Tu_{\delta}+\int\chi(dd^{c}\phi_{M-1}-dd^{c}\phi_{M})\wedge T(u-u_{\delta}). \\
\end{split}
\end{equation*}
Since
\begin{equation*}
\left|\int\chi dd^{c}(\phi T)\right|=\left|\int dd^{c}\chi\wedge\phi T\right|\leq k\|\chi\|_{\mathscr{C}^{2}}\int_{B_{1-r}}|\phi|d\sigma_{T},
\end{equation*}
We obtain
\begin{equation*}
\begin{split}
&\left|\int\chi(dd^{c}\phi_{M-1}-dd^{c}\phi_{M})\wedge T\right|
\leq 2k\|\chi\|_{\mathscr{C}^{2}}\int_{B_{1-r}}|\phi|d\sigma_{T} \\
&\leq 2k\|\chi\|_{\mathscr{C}^{2}}\frac{1}{\alpha}\int_{B_{1-r}}\exp(-\alpha\phi)d\sigma_{T}\leq 2c_{3}k\frac{c}{\alpha}. \\
\end{split}
\end{equation*}
Then
\begin{equation}
\left|\int\chi (dd^{c}\phi_{M-1}-dd^{c}\phi_{M})\wedge T(u-u_{\delta})\right|\leq 2c_{3}k\frac{c}{\alpha}\epsilon\delta^{\rho}
\end{equation}
Using Lemma 2.9 again, we obtain
\begin{equation*}
\begin{split}
&\int\chi dd^{c}\psi_{M}\wedge Tu_{\delta} \\
&=\int_{B_{1-r}\setminus B_{1-3r}}dd^{c}\chi\wedge\psi_{M}Tu_{\delta}+\int_{B_{1-r}\setminus B_{1-3r}}d\chi\wedge\psi_{M}T\wedge d^{c}u_{\delta} \\
&-\int_{B_{1-r}\setminus B_{1-3r}}d^{c}\chi\wedge\psi_{M}T\wedge du_{\delta}+\int_{B_{1-r}}\chi\psi_{M}T\wedge dd^{c}u_{\delta}. \\
\end{split}
\end{equation*}
By the same argument, the first three integrals have the same dominant constant
\begin{equation}
c_{4}\epsilon kce^{\alpha}e^{-\alpha M}.
\end{equation}
The final term
\begin{equation}
\begin{split}
&\left|\int\chi\psi_{M}T\wedge dd^{c}u_{\delta}\right|\leq ce^{\alpha}e^{-\alpha M}\|u_{\delta}\|_{\mathscr{C}^{2}} \\
&\leq ce^{\alpha}e^{-\alpha M}\epsilon\delta^{-(2-\rho)}. \\
\end{split}
\end{equation}
Let $\delta=e^{-\alpha M/2}$ small enough, since it is sufficient to consider $M$ big.
Then $e^{-\alpha M\rho/2}= e^{-\alpha M}e^{\alpha M(2-\rho)/2}$.
Combining $(4),(5),(6),(7)$,
we have
\begin{equation*}
\sigma_{T^{'}}\{z\in B_{1}, \phi <-M\}
\leq \epsilon ck(2c_{4}e^{\alpha}+\frac{e^{\alpha}}{k}+2\frac{c_{3}}{\alpha})e^{-\frac{\alpha M}{2}\rho}.
\end{equation*}
So by Remark 2.5 we have
\begin{equation*}
\begin{split}
&\int_{B_{1}}\exp(-\frac{\alpha\rho}{4}\phi)dd^{c}(uT) \\
&\leq 2\epsilon ck(2c_{4}e^{\alpha}+\frac{e^{\alpha}}{k}+2\frac{c_{3}}{\alpha})
\leq \epsilon ck(c_{1}e^{\alpha}+\frac{c_{2}}{\alpha}), \\
\end{split}
\end{equation*}
where $c_{1}=4c_{4}+2, c_{2}=4c_{3}$.
\end{proof}

The following proposition is our main result about the estimate for moderate measures on $\mathbb{P}^{k}$.
\begin{proposition}
Suppose that $u_{j}$ is of class $\mathscr{C}^{\rho}$ with modulus $\epsilon$ on $\mathbb{P}^{k}$ for some $0<\rho<1, 0<\epsilon<1$,
and that $u_{j}$ is an $\epsilon\omega_{FS}$-p.s.h. function for all $1\leq j\leq k$.
Assume that $\epsilon<\beta_{0}k^{-3}(\frac{\rho}{12})^{2k}$,
where $\beta_{0}$ is a positive constant independent of $k$ and $\rho$.
Then there exists a positive constant $c_{5}$ independent of $k$ and $\rho$, such that
\begin{equation}
\int_{\mathbb{P}^{k}}\exp(-\alpha_{0}(\frac{\rho}{4})^{k}\phi)
(\wedge_{j=1}^{j=k}(dd^{c}u_{j}+\epsilon\omega_{FS}+\omega_{FS})-\omega_{FS}^{k})\leq c_{5}(\frac{\rho}{4})^{k}
\end{equation}
for all $\phi\in\mathcal{F}$, where $\alpha_{0}$ is the constant in Proposition 2.2.
In other words, $(\wedge_{j=1}^{j=k}(dd^{c}u_{j}+\epsilon\omega_{FS}+\omega_{FS})-\omega_{FS}^{k})$ is
$(c_{5}(\frac{\rho}{4})^{k}, \alpha_{0}(\frac{\rho}{4})^{k})$-moderate.
\end{proposition}
\begin{proof}
We pull back the integral (8) locally to that on $\mathbb{C}^{k}$.
There is a potential $v=\frac{\epsilon}{2}\log(1+\|z\|^{2})$ on $\mathbb{C}^{k}$ such that $(\theta^{-1})^{\star}(\epsilon\omega_{FS})=dd^{c}v$,
where the map $\theta$ is defined in (2).
Set $\tilde u_{j}:=u_{j}\circ\theta^{-1}+v$. Note that $u_{j}$ is $\epsilon\omega_{FS}$-p.s.h., then $dd^{c}\tilde u_{j}\geq 0$.
Since $u_{j}$ is of class $\mathscr{C}^{\rho}$ with modulus $\epsilon$ on $\mathbb{P}^{k}$,
$\log(1+\|z\|^{2})$ is of class $\mathscr{C}^{a}$ on $\mathbb{C}^{k}$ for all $0<a<1$,
then we may assume that $\tilde u_{j}$ is of class $\mathscr{C}^{\rho}$ with modulus $\epsilon$ on $B_{1}$.
Hence $\|\tilde u_{j}\|_{\mathscr{C}^{\rho}(B_{1})}\leq\epsilon$.
Let $\omega=\frac{1}{2}dd^{c}\log(1+\|z\|^{2})$, we have
\begin{equation}
\begin{split}
& \int_{K_{0}}\exp(-\alpha\phi)\wedge_{j=1}^{j=k}(dd^{c}u_{j}+\epsilon\omega_{FS}+\omega_{FS}) \\
& =\int_{B_{1}}\exp(-\alpha\phi\circ\theta^{-1})(\theta^{-1})^{\star}\wedge_{j=1}^{j=k}(dd^{c}u_{j}+\epsilon\omega_{FS}+\omega_{FS}) \\
& =\int_{B_{1}}\exp(-\alpha\phi\circ\theta^{-1})(dd^{c}\tilde u_{1}+\omega)\wedge\cdot\cdot\cdot\wedge(dd^{c}\tilde u_{k}+\omega) \\
\end{split}
\end{equation}
We replace $\tilde u_{j}$ (resp. $\phi\circ\theta^{-1}$) by $u_{j}$ (resp. $\phi$) in the sequel.
Since there are two constants $c_{0}>0, \alpha_{0}>0$ independent of $k$ and $\rho$, such that
\begin{equation}
\int_{\mathbb{P}^{k}}\exp(-\alpha_{0}\tilde\phi)\omega_{FS}^{k}\leq c_{0}k,
\end{equation}
by pulling back the integral in $B_{1}$ with Lemma 2.10, we have
\begin{equation*}
\int_{B_{1}}\exp(-\alpha_{0}\frac{\rho}{4}\phi)(dd^{c}u_{j})\wedge\omega^{k-1}
\leq c_{0}\epsilon k^{2}(c_{1}e^{\alpha_{0}}+\frac{c_{2}}{\alpha_{0}}).
\end{equation*}
By induction we can show that
\begin{equation*}
\int_{B_{1}}\exp(-\alpha_{0}(\frac{\rho}{4})^{j}\phi)dd^{c}u_{l_{1}}\wedge\cdot\cdot\cdot\wedge dd^{c}u_{l_{j}}\wedge\omega^{k-j}
\leq c_{0}k(\epsilon k)^{j}\prod_{l=0}^{j-1}(c_{1}e^{\alpha_{0}(\frac{\rho}{4})^{l}}+\frac{c_{2}}{\alpha_{0}(\frac{\rho}{4})^{l}})
\end{equation*}
for all $1\leq l_{1}<\cdot\cdot\cdot<l_{j}\leq k$.
Let $\beta_{0}=1/(c_{1}e^{\alpha_{0}}+\frac{c_{2}}{\alpha_{0}}), \epsilon_{0}=\beta_{0}k^{-3}(\frac{1}{8})^{k}(\frac{\rho}{4})^{\frac{3k-1}{2}}>\epsilon, \epsilon_{0} =\epsilon_{1}\epsilon_{2},\epsilon_{2}=(\frac{\rho}{4})^{k}$.
Here $\beta_{0}$ is independent of $k$ and $\rho$.
Let $\epsilon_{1}=\epsilon_{3}/(\frac{\rho}{4})^{\frac{k+1}{2}}$, then $\epsilon_{3}=\beta_{0}(\frac{\rho}{32})^{k}/k^{3}$. Hence
\begin{equation}
\begin{split}
&\int_{B_{1}}\exp(-\alpha_{0}(\frac{\rho}{4})^{k}\phi)((dd^{c}u_{1}+\omega)\wedge\cdot\cdot\cdot\wedge(dd^{c}u_{k}+\omega)-\omega^{k}) \\
&=\sum_{j=1}^{k}\binom{k}{j}\int_{B_{1}}\exp(-\alpha_{0}(\frac{\rho}{4})^{j}\phi)dd^{c}u_{1}\wedge\cdot\cdot\cdot\wedge dd^{c}u_{j}\wedge\omega^{k-j} \\
&\leq \sum_{j=1}^{k}\binom{k}{j} c_{0}k(\epsilon_{1}k)^{j}(\frac{1}{\beta_{0}})^{j}(\frac{\rho}{4})^{k+k-1+\cdots+k-(j-1)} \\
&\leq \sum_{j=1}^{k}\binom{k}{j} c_{0}k(\epsilon_{1}k)^{j}(\frac{1}{\beta_{0}})^{j}(\frac{\rho}{4})^{\frac{k+1}{2}j}
\leq c_{0}k\sum_{j=1}^{k}\binom{k}{j}(\frac{\epsilon_{3}k}{\beta_{0}})^{j} \\
&\leq c_{0}(\frac{\rho}{32})^{k}(\sum_{j=0}^{k-1}\frac{1}{k^{j}})\leq 2c_{0}(\frac{\rho}{32})^{k}. \\
\end{split}
\end{equation}
This is equivalent to
\begin{equation*}
\int_{K_{0}}\exp(-\alpha_{0}(\frac{\rho}{4})^{k}\phi)
((dd^{c}u_{1}+\epsilon\omega_{FS}+\omega_{FS})
\wedge\cdot\cdot\cdot\wedge(dd^{c}u_{k}+\epsilon\omega_{FS}+\omega_{FS})-\omega_{FS}^{k})\leq 2c_{0}(\frac{\rho}{32})^{k}.
\end{equation*}
By Lemma 2.8, there is a positive constant $N^{\prime}$ independent of $k$ and $\rho$ such that $N_{k}\leq N^{\prime}8^{k}$.
Let $c_{5}=2c_{0}N^{\prime}$. Due to the homogeneity of $\mathbb{P}^{k}$, we have
\begin{equation*}
\int_{\mathbb{P}^{k}}\exp(-\alpha_{0}(\frac{\rho}{4})^{k}\phi)
((dd^{c}u_{1}+\epsilon\omega_{FS}+\omega_{FS})\wedge\cdot\cdot\cdot\wedge(dd^{c}u_{k}+\epsilon\omega_{FS}+\omega_{FS})-\omega_{FS}^{k})\leq c_{5}(\frac{\rho}{4})^{k}.
\end{equation*}
The proof is completed.
\end{proof}

\begin{remark}
Since $(dd^{c}u_{j}+\omega_{FS})^{k}\leq(dd^{c}u_{j}+\epsilon\omega_{FS}+\omega_{FS})^{k}$, the above proposition,
combined with (10), gives the following estimate
\begin{equation*}
\int_{\mathbb{P}^{k}}\exp(-\alpha_{0}(\frac{\rho}{4})^{k}\phi)(dd^{c}u_{1}+\omega_{FS})\wedge\cdot\cdot\cdot\wedge(dd^{c}u_{k}+\omega_{FS})\leq c_{0}k+c_{5}(\frac{\rho}{4})^{k}\leq c_{0}k+c_{5}
\end{equation*}
for all $\phi\in\mathcal{F}$.
In other words, $(dd^{c}u_{1}+\omega_{FS})\wedge\cdot\cdot\cdot\wedge(dd^{c}u_{k}+\omega_{FS})$ is $(c_{0}k+c_{5}, \alpha_{0}(\frac{\rho}{4})^{k})$-moderate.
\end{remark}

\section{Zeros of sections of ample line bundles}
In this section, we will prove the main theorem.
Consider a projective manifold $X$ of dimension $k$ and an ample line bundle $L$ on $X$.
There exists a smooth Hermitian metric $h$ such that
\begin{equation*}
c_{1}(h)=-dd^{c}\log h(e_{L},e_{L})^{\frac{1}{2}}
\end{equation*}
is a strictly positive $(1,1)$-form, where $e_{L}$ is a local holomorphic section on $L$.
As we know, $c_{1}(h)$ represents the Chern class $c_{1}(L)\in H^{2}(X,\mathbb{Z})$.
Let $\omega =c_{1}(h)$ be the K\"{a}hler form, $\int_{X}\omega^{k}=c_{1}(L)^{k}\in\mathbb{Z}^{+}$.
The line bundle $L^{n}$ of the $n$th tensor power of $L$ has a natural Hermitian metric $h_{n}$ induced by $h$.
The space $H^{0}(X,L^{n})$ of holomorphic sections of $L^{n}$ has the following inner product,
\begin{equation*}
\langle s_{1},s_{2}\rangle: =\frac{1}{c_{1}(L)^{k}}\int_{X}h_{n}(s_{1},s_{2})\omega^{k}
\end{equation*}
$\forall s_{1},s_{2}\in H^{0}(X,L^{n})$. For more details, see \cite{dj}

Denote by $\mathbb{P}H^{0}(X,L^{n})^{\star}$ the dual of $\mathbb{P}H^{0}(X,L^{n})$.
Let $\omega_{FS}$ be the standard normalized Fubini-Study form with no confusion.
When $n$ is big enough, the Kodaira map is defined by
\begin{equation*}
\begin{split}
& \Phi_{n}: X\rightarrow\mathbb{P}H^{0}(X,L^{n})^{\star}, \\
& \Phi_{n}(x):=\{s\in\mathbb{P}H^{0}(X,L^{n}): s(x)=0\}. \\
\end{split}
\end{equation*}
Note that $\Phi_{n}(x)$ can be regarded as a hyperplane in $\mathbb{P}H^{0}(X,L^{n})$.
Choose an orthonormal basis $\{s_{n,j}\}_{j=0}^{k_{n}}$ with respect to the above inner product on $H^{0}(X,L^{n})$.
Then by an identification via the basis, we obtain a holomorphic map
\begin{equation*}
\Phi_{n}: X\rightarrow\mathbb{P}^{k_{n}}.
\end{equation*}
Let $U\subset X$ be a contractible Stein open subset, $e_{L}$ a local holomorphic frame of $L$ on $U$.
Then there exist holomorphic functions $\tilde s_{n,j}$ on $U$ such that $s_{n,j}=\tilde s_{n,j}e_{L}^{\otimes n}$.
Then the map is expressed locally as follows,
\begin{equation*}
\Phi_{n}(x)=[\tilde s_{n,0}(x):...:\tilde s_{n,k_{n}}(x)], \quad \forall x\in U.
\end{equation*}
We call $\Phi_{n}^{\star}(\omega_{FS})$ the Fubini-Study current which is independent of the choice of basis.
Recall that a meromorphic transform between two complex manifolds is a surjective multivalued map with an analytic graph.
To be more precise, let $(X_{1},\omega_{1}), (X_{2},\omega_{2})$ be two compact K\"{a}hler manifolds of dimension $n_{1}$ and $n_{2}$ respectively,
a meromorphic transform $F: X_{1}\rightarrow X_{2}$ is the data of an analytic subset $\Gamma\subset X_{1}\times X_{2}$ of pure dimension
$n_{2}+l$ such that the natural projections $\pi_{1}: X_{1}\times X_{2}\rightarrow X_{1}$ and $\pi_{2}: X_{1}\times X_{2}\rightarrow X_{2}$
restricted to each irreducible component of $\Gamma$ are surjective.
$\Gamma$ is called the graph of $F$.
We write $F=\pi_{2}\circ (\pi_{1}|_{\Gamma})^{-1}$.
The dimension of the fiber $F^{-1}(x_{2}):=\pi_{1}(\pi_{2}^{-1}|_{\Gamma}(x_{2}))$ is equal to $l$
for the point $x_{2}\in X_{2}$ generic. This is the codimension of the meromorphic transform $F$.
If $T$ is a current of bidegree $(m,m)$ on $X_{2}$, $n_{2}+l-n_{1}\leq m\leq n_{2}$, we define $F^{\star}(T):=(\pi_{1})_{\star}(\pi_{2}^{\star}(T)\wedge [\Gamma])$,
where $[\Gamma]$ is the current of integration over $\Gamma$.
The intermediate degree of order $m$ of a meromorphic transform $F: X_{1}\rightarrow X_{2}$ is defined by
\begin{equation*}
\lambda_{m}(F)=\int_{X_{1}}F^{\star}(\omega_{2}^{m})\wedge\omega^{n_{2}+l-m}_{1}
=\int_{X_{2}}\omega^{m}_{2}\wedge F_{\star}(\omega^{n_{2}+l-m}_{1}).
\end{equation*}

Now we consider the meromorphic transforms from $X$ to $\mathbb{P}H^{0}(X,L^{n})$ induced by the Kodaira maps.
The meromorphic transform $F_{n}: X\rightarrow \mathbb{P}H^{0}(X,L^{n})$ has the following graph
\begin{equation*}
\Gamma_{n}=\{(x,s)\in X\times\mathbb{P}H^{0}(X,L^{n}): s(x)=0\}.
\end{equation*}
Since $X$ is ample,
for every point $x\in X$, there exists a point $s\in\mathbb{P}H^{0}(X,L^{n})$ such that $s(x)=0$.
Hence the projection from $\Gamma_{n}$ to $X$ is surjective.
Since $L^{n}$ is not trivial, there are no nowhere vanishing sections.
That is to say, every point $s\in\mathbb{P}H^{0}(X,L^{n})$ must vanish at some point $x\in X$.
Hence the projection from $\Gamma_{n}$ to $\mathbb{P}H^{0}(X,L^{n})$ is surjective.
Then $F_{n}$ is indeed a meromorphic transform of codimension $k-1$.
For more details about these meromorphic transforms, refer to \cite[Example 3.6(c)]{ds1}.
Note that $\delta_{n}:=\lambda_{k_{n}-1}(F_{n})$ (resp. $d_{n}:=\lambda_{k_{n}}(F_{n})$)
is the intermediate degree of order $k_{n}-1$ (resp. $k_{n}$) of $F_{n}$.
\begin{lemma}
In the above setting, $\delta_{n}$ is bounded and $d_{n}=nc_{1}(L)^{k}$.
Moreover, $F_{n}^{\star}(\omega_{FS}^{k_{n}})=\Phi_{n}^{\star}(\omega_{FS})$.
\end{lemma}
\begin{proof}
The first assertion is proved in \cite[Lemma 7.1]{ds1} by using cohomological arguments.
We prove the second one with the definition of $F_{n}^{\star}$.
For any test $(k-1,k-1)$-form $\psi$, we have
\begin{equation*}
\bigl<F_{n}^{\star}(\omega_{FS}^{k_{n}}), \psi\bigr>=\int_{\Gamma_{n}}\pi_{1}^{\star}(\psi)\wedge\pi_{2}^{\star}(\omega_{FS}^{k_{n}})
\end{equation*}
\begin{equation*}
\begin{split}
& =\int_{\mathbb{P}H^{0}(X,L^{n})}\pi_{2\star}\pi_{1}^{\star}(\psi)\wedge\omega_{FS}^{k_{n}} \\
& =\int_{\mathbb{P}H^{0}(X,L^{n})}\int_{\pi_{2}^{-1}(s_{n})\cap\Gamma_{n}}\pi_{1}^{\star}(\psi)\omega_{FS}^{k_{n}}(s_{n}) \\
& =\int_{\mathbb{P}H^{0}(X,L^{n})}\int_{\{x\in X: s_{n}(x)=0\}}\psi \omega_{FS}^{k_{n}}(s_{n}) \\
& =\int_{\mathbb{P}H^{0}(X,L^{n})}\bigl<[\bm{Z}_{\bm{s_{n}}}], \psi \bigr> \omega_{FS}^{k_{n}}(s_{n}) \\
& =\bigl<\Phi_{n}^{\star}(\omega_{FS}), \psi\bigr>. \\
\end{split}
\end{equation*}
The last equality follows from \cite[Proposition 4.2]{cm}.
This completes the proof.
\end{proof}

From now on we introduce some other notations and properties from \cite{ds1}.
Suppose that $\mu$ is a PLB probability measure on $\mathbb{P}^{k}$. $\mathcal{F}$ is defined in (1) when $X=\mathbb{P}^{k}$.
Let
\begin{equation*}
\begin{split}
& Q(\mathbb{P}^{k},\omega_{FS})=\{\phi ~q.p.s.h. ~on~ \mathbb{P}^{k}: dd^{c}\phi\geq -\omega_{FS}\}, \\
& R(\mathbb{P}^{k}, \omega_{FS},\mu)=\sup_{\phi}\left\{-\int\phi d\mu , \phi\in\mathcal{F}\right\}, \\
\end{split}
\end{equation*}
\begin{equation*}
\begin{split}
& S(\mathbb{P}^{k}, \omega_{FS},\mu)=\sup_{\phi}\left\{\left|\int\phi d\mu\right| , \phi\in Q(\mathbb{P}^{k},\omega_{FS}), \int\phi\omega_{FS}^{k}=0\right\}, \\
& \Delta(\mathbb{P}^{k}, \omega_{FS},\mu ,t)=\sup_{\phi}\left\{\mu (\phi<-t) , \phi\in Q(\mathbb{P}^{k},\omega_{FS}), \int\phi d\mu =0\right\} \\
\end{split}
\end{equation*}
for any $t>0$.
When $\mu =\omega_{FS}^{k}$, we write $R^{0}(\mathbb{P}^{k}, \omega_{FS})=R(\mathbb{P}^{k}, \omega_{FS},\mu)$.
These constants are related to Alexander-Dinh-Sibony capacity \cite{ds1}.
\begin{proposition}
$S(\mathbb{P}^{k}, \omega_{FS},\mu)\leq R(\mathbb{P}^{k}, \omega_{FS},\mu)+R^{0}(\mathbb{P}^{k}, \omega_{FS})$.
\end{proposition}
The above proposition comes from Section 2 in \cite{ds1}.
There is an important estimate for $R^{0}(\mathbb{P}^{k}, \omega_{FS})$, see \cite[Proposition A.3]{ds1}.
\begin{proposition}
\begin{equation*}
R^{0}(\mathbb{P}^{k}, \omega_{FS})\leq \frac{1}{2}(1+\log k) .
\end{equation*}
\end{proposition}
Let $\sigma_{n}$ be a PLB probability measure on $\mathbb{P}H^{0}(X,L^{n})$.
To simplify the notations, let
\begin{equation*}
\begin{split}
&R_{n}:=R(\mathbb{P}H^{0}(X,L^{n}),\omega_{FS},\sigma_{n}), \\
&R^{0}_{n}:=R(\mathbb{P}H^{0}(X,L^{n}),\omega_{FS},\omega^{k_{n}}_{FS}), \\
&S_{n}:=S(\mathbb{P}H^{0}(X,L^{n}),\omega_{FS},\sigma_{n}), \\
&\Delta_{n}(t):=\Delta(\mathbb{P}H^{0}(X,L^{n}),\omega_{FS},\sigma_{n},t). \\
\end{split}
\end{equation*}
Let $\mathbb{P}^{X}:=\Pi_{n\geq 1}\mathbb{P}H^{0}(X,L^{n})$ endowed with its measure
$\sigma=\Pi_{n\geq 1}\sigma_{n}$. Denote by $\delta_{z}$ the Dirac measure at a point $z$.
We specify the following two theorems for the above case, see \cite{ds1}.
\begin{theorem}
Suppose that the sequence $\{R_{n}\delta_{n}d^{-1}_{n}\}$ tends to $0$ and
\begin{equation*}
\Sigma_{n\geq 1}\Delta_{n}(\delta^{-1}_{n}d_{n}t)<\infty
\end{equation*}
for all $t>0$.
Then for almost everywhere $s=(s_{n})\in\mathbb{P}^{X}$ with respect to $\sigma$, the sequence $\langle d_{n}^{-1}(F_{n}^{\star}(\delta_{s_{n}})-F_{n}^{\star}(\sigma_{n})), \psi \rangle$
converges to $0$ uniformly on the bounded set of $(k-1,k-1)$-forms on $X$ of class $\mathscr{C}^{2}$.
\end{theorem}
\begin{theorem}
Suppose that the sequence $\{S_{n}\delta_{n}d^{-1}_{n}\}$ tends to $0$.
Then $\langle d_{n}^{-1}(F_{n}^{\star}(\sigma_{n})$ $-F_{n}^{\star}(\omega^{k_{n}}_{FS})), \psi \rangle$ converges to $0$ uniformly on the bounded set of $(k-1,k-1)$-forms on $X$ of class $\mathscr{C}^{2}$.
\end{theorem}
In fact, Dinh and Sibony proved the above two theorems for any countable family of compact K\"{a}hler manifolds with meromorphic transformations.
The following theorem is due to Tian-Zelditch \cite{zs}.
\begin{theorem}
For all $r\geq 0$, $\|n^{-1}\Phi^{\star}_{n}(\omega_{FS})-\omega\|_{\mathscr{C}^{r}}=O(n^{-1})$.
\end{theorem}
In order to prove the main theorem,
we write
\begin{equation*}
\begin{split}
& |\bigl<n^{-1}[\bm{Z}_{\bm{s_{n}}}]-\omega, \psi\bigr>|\leq|\bigl<n^{-1}[\bm{Z}_{\bm{s_{n}}}]-n^{-1}F^{\star}_{n}(\sigma_{n}), \psi\bigr>| \\
& +|\bigl<n^{-1}F^{\star}_{n}(\sigma_{n})-n^{-1}F^{\star}_{n}(\omega_{FS}^{k_{n}}), \psi\bigr>|+
|\bigl<n^{-1}F^{\star}_{n}(\omega_{FS}^{k_{n}})-\omega, \psi\bigr>|, \\
\end{split}
\end{equation*}
for any test form $\psi$ of bidegree $(k-1,k-1)$ on $X$.
It is sufficient to prove that the three terms in the right side of the inequality all tend to $0$ when $n\rightarrow\infty$.
The third one is right due to Theorem 3.6.
The first one holds under the conditions that $R_{n}=o(n), \quad \sum_{n\geq 1}\Delta(nt)<\infty, \quad \forall t>0$ by Theorem 3.4.
The second one is valid when $S_{n}=o(n)$ by Theorem 3.5.
By applying Proposition 3.2 and Proposition 3.3,
the proof is reduced to
the estimates of $R_{n}/n$ and $\sum_{n\geq 1}\Delta(nt)$ for any $t>0$. \\

\noindent $\textbf{End of the proof of Theorem 1.1}$.
We have $F^{\star}_{n}(\omega_{FS}^{k_{n}})=(\Phi_{n})^{\star}\omega_{FS}$ by Lemma 3.1.
It follows from Theorem 3.6 that
\begin{equation}
n^{-1}F^{\star}_{n}(\omega_{FS}^{k_{n}})\rightarrow\omega
\end{equation}
in the weak sense of currents.
We write $\mu_{1,n}=\omega_{FS}^{k_{n}},\mu_{2,n}=\wedge_{j=1}^{k_{n}}(dd^{c}u_{n,j}+\epsilon_{n}\omega_{FS}+\omega_{FS})-\mu_{1,n}$.
Then $\sigma_{n}\leq \mu_{1,n}+\mu_{2,n}$.
Note that $k_{n}=c_{1}(L)^{k}n^{k}/k!+O(n^{k-1})$.
Let $c>(\frac{12}{\rho})^{2c_{1}(L)^{k}/k!}>1$ such that $c^{n^{k}}\geq\frac{1}{\beta_{0}}k_{n}^{3}(\frac{12}{\rho})^{2k_{n}}$,
then $c$ depends only on $X, L$ and $\rho$. Hence $\mu_{2,n}$ is a positive moderate measure satisfying Proposition 2.11.
To estimate $\Delta_{n}$, we consider any q.p.s.h. function $\phi$ on $\mathbb{P}^{k_{n}}$ such that $dd^{c}\phi\geq -\omega_{FS}$
and $\int\phi d\sigma_{n}=0$. Set $\varphi:=\phi-\max_{\mathbb{P}^{k_{n}}}\phi$.
It is obvious that $\varphi\in\mathcal{F}$ by definition in (1). Since $\int\phi d\sigma_{n}=0$,
$\max_{\mathbb{P}^{k_{n}}}\phi\geq 0$. Hence $\varphi\leq\phi$. Then we have
\begin{equation*}
\begin{split}
& \sigma_{n}(\phi<-nt)\leq\sigma_{n}(\varphi<-nt) \\
& \leq\mu_{1,n}(\varphi<-nt)+\mu_{2,n}(\varphi<-nt) \\
& \leq\int\exp(\alpha_{0}(-nt-\varphi))d\mu_{1,n}+\int\exp(\alpha_{0}(\frac{\rho}{4})^{k_{n}}(-nt-\varphi))d\mu_{2,n} \\
& \leq c_{0}k_{n}\exp(-\alpha_{0}nt)+c_{5}(\frac{\rho}{4})^{k_{n}}\exp(-\alpha_{0}(\frac{\rho}{4})^{k_{n}}nt). \\
\end{split}
\end{equation*}
The last inequality follows from Proposition 2.2 and Proposition 2.11.
Then by the definition of $\Delta_{n}$, we have
\begin{equation}
\sum_{n\geq 1}\Delta_{n}(nt)
\leq\sum_{n\geq 1}c_{0}k_{n}\exp(-\alpha_{0}nt)+\sum_{n\geq 1}c_{5}(\frac{\rho}{4})^{k_{n}}\exp(-\alpha_{0}(\frac{\rho}{4})^{k_{n}}nt).
\end{equation}
It is obvious that $\sum_{n\geq 1}n^{k}\exp (-nt)<\infty$ and that
$\exp(-(\frac{\rho}{4})^{k_{n}}nt)$ tends to $1$ when $n$ tends to infinity, $\forall t>0$.
This yields $\sum_{n\geq 1}\Delta_{n}(nt)<\infty$.
By Proposition 3.3 and Proposition 2.11,
\begin{equation}
\limsup_{n\to\infty}R_{n}^{0}/n\leq\lim_{n\to\infty}\frac{1+\log k_{n}}{2n}=0.
\end{equation}
\begin{equation}
\begin{split}
&\limsup_{n\to\infty}R_{n}/n\leq\lim_{n\to\infty}\sup_{\phi\in\mathcal{F}}\left\{-\int\phi d\mu_{1,n}-\int\phi d\mu_{2,n}\right\}/n \\
&\leq\limsup_{n\to\infty}R_{n}^{0}/n+\lim_{n\to\infty}c_{5}(\frac{\rho}{4})^{k_{n}}/(\alpha_{0}(\frac{\rho}{4})^{k_{n}}n)=0  \\
\end{split}
\end{equation}
By Proposition 3.2, (14) and (15), $\limsup_{n\to\infty}S_{n}/n=0$.
Note that $\delta_{n}d_{n}^{-1}=O(\frac{1}{n})$ by Lemma 3.1.
Hence by applying Theorem 3.5, the following sequence
\begin{equation}
n^{-1}F^{\star}_{n}(\sigma_{n})-n^{-1}F^{\star}_{n}(\omega_{FS}^{k_{n}})\rightarrow 0
\end{equation}
in the weak sense of currents.
We know that $F_{n}^{\star}(\delta_{s_{n}})=[\bf{Z}_{s_{n}}]$ by the definition of $F_{n}^{\star}$.
Combined with (13) and (15), Theorem 3.4 implies that for $\sigma$-almost everywhere $s\in\mathbb{P}^{X}$,
the following sequence
\begin{equation}
n^{-1}[\bm{Z}_{\bm{s_{n}}}]-n^{-1}F^{\star}_{n}(\sigma_{n})\rightarrow 0
\end{equation}
in the weak sense of currents.
Then we deduce from (12), (16) and (17) that for $\sigma$-almost everywhere $s\in\mathbb{P}^{X}$,
\begin{equation*}
\begin{split}
& |\bigl<n^{-1}[\bm{Z}_{\bm{s_{n}}}]-\omega, \psi\bigr>|\leq|\bigl<n^{-1}[\bm{Z}_{\bm{s_{n}}}]-n^{-1}F^{\star}_{n}(\sigma_{n}), \psi\bigr>| \\
& +|\bigl<n^{-1}F^{\star}_{n}(\sigma_{n})-n^{-1}F^{\star}_{n}(\omega_{FS}^{k_{n}}), \psi\bigr>|+
|\bigl<n^{-1}F^{\star}_{n}(\omega_{FS}^{k_{n}})-\omega, \psi\bigr>|\rightarrow 0, \\
\end{split}
\end{equation*}
for any test form $\psi$ of bidegree $(k-1,k-1)$ on $X$ when $n$ tends to $\infty$.
That is to say, $n^{-1}[\bf{Z}_{s_{n}}]$ converges weakly to $\omega$.
$\vspace{12pt}$
The proof is completed. \qquad \qquad \quad$\qed$

Now given $X$ and $L$ in Theorem 1.1, we construct a concrete example of a sequence of functions $(u_{n,j})$
satisfying the conditions of the theorem. We require that $u_{n,1}=\cdot\cdot\cdot=u_{n,k_{n}}=u_{n}$.
Notice that we can perturbate $u_{n}$ so that the constants $\xi_{n}, \epsilon_{n}$ do not change and the perturbed functions still
satisfy the conditions in Theorem 1.1.
\begin{example}
\emph{
Let $\pi: \mathbb{C}^{k+1}\setminus\{0\}\rightarrow\mathbb{P}^{k}$ be the natural map.
Consider the map $f:\mathbb{P}^{k}\rightarrow\mathbb{P}^{k}$ with $f[z_{0},...,z_{k}]=[z_{0}^{k},...,z_{k}^{k}]$.
From \cite[Example 1.6.4]{s}, its Green function is $s(z)=\max(\log|z_{0}|,...,\log|z_{k}|)$.
Moreover, $s$ is a H\"{o}lder continuous function with any exponent $0<\rho<1$.
We obtain a well-defined function
\begin{equation}
v:=\max(\log\frac{|z_{0}|}{|z|},...,\log\frac{|z_{k}|}{|z|})
\end{equation}
on $\mathbb{P}^{k}$.
Since $\pi^{\star}(dd^{c}v+\omega_{FS})=dd^{c}s\geq 0$, then $v$ is $\omega_{FS}$-p.s.h. and H\"{o}lder continuous with any exponent $0<\rho<1$.
Denote by $d_{FS}$ the distance induced by Fubini-Study metric.
Let $d_{k}=\sup_{\substack{z,w\in\mathbb{P}^{k} \\z\neq w}}\frac{|v(z)-v(w)|}{d_{FS}(z,w)^{\rho}}$.
We will show that
\begin{equation}
d_{k}\leq\sqrt{\pi}k
\end{equation}
at the end of the example.
For each $n$, we obtain a corresponding function $v_{n}$ using (18) and identifying $\mathbb{P}H^{0}(X,L^{n})$ with $\mathbb{P}^{k_{n}}$.
Consider the functions $u_{n}=c_{n}'v_{n}$ with suitable constants $c_{n}'=O(\frac{1}{n^{k}c^{n^{k}}})<1/c^{n^{k}}$, where $c=(145)^{c_{1}(L)^{k}/k!}$.
Let $\epsilon_{n}:=c_{n}'$. Since $k_{n}=O(n^{k})$, it follows from (11) that $d_{k_{n}}=O(n^{k})$. Consequently, $u_{n}$ is of class $\mathscr{C}^{\rho}$ with modulus $1/c^{n^{k}}$. Moreover, since $v_{n}$ is $\omega_{FS}$-p.s.h., we infer that $u_{n}$ is $\epsilon_{n}\omega_{FS}$-p.s.h..
So $\{u_{n}\}$ satisfy the three conditions in Theorem 1.1.
From the above proof, we see that $\sigma=\prod_{n\geq 1}\sigma_{n}=(dd^{c}u_{n}+\omega_{FS})^{k_{n}}$
satisfies the equidistribution property. }

\emph{Finally we prove (19).
It is sufficient to consider the special case when $|z_{0}|\geq\max\{|z_{1}|,...,|z_{k}|\}, |w_{0}|\geq\max\{|w_{1}|,...,|w_{k}|\}$.
Then
\begin{equation*}
d_{k}=\frac{1}{2}\sup_{\substack{z,w\in K \\z\neq w}}\frac{\bigl |\log(1+|z|^{2})-\log(1+|w|^{2})\bigr |}{d_{FS}(z,w)^{\rho}}
\end{equation*}
where $z=(\frac{z_{1}}{z_{0}},...,\frac{z_{1}}{z_{0}}), w=(\frac{w_{1}}{w_{0}},...,\frac{w_{1}}{w_{0}})\in\mathbb{C}^{k}$
and $K=\{z\in\mathbb{C}^{k}: |z_{i}|\leq 1, 1\leq i\leq k\}$.
Let $g=\sum_{i,j=1}^{2k}g_{ij}dx^{i}\otimes dx^{j}$ be the associated Riemannian metric with
$g_{11}=\frac{1}{\pi}\frac{1+|z|^{2}-|z_{1}|^{2}}{(1+|z|^{2})^{2}}$.
When $r_{1}=|z|, r_{2}=|w|$ are fixed, $d_{FS}(z,w)$ takes its minimum only when $z$ and $w$ are at the same line through the origin in $\mathbb{R}^{2k}$.
The distance is invariant with respect to the orthogonal group $O(2k)$ in this case since the Fubini-Study metric is invariant with respect to the unitary group $U(k)$ on $\mathbb{P}^{k}$.
So we take the simple case when $z=(r_{1},0,...,0), w=(r_{2},0,...,0)$.
Hence
\begin{equation*}
\begin{split}
& d_{k}= \frac{\sqrt{\pi}}{2}\sup_{0\leq r_{1}<r_{2}\leq k}\frac{\log(1+r_{2}^{2})-\log(1+r_{1}^{2})}{(\arctan r_{2}-\arctan r_{1})^{\rho}} \\
&=\frac{\sqrt{\pi}}{2}\sup_{0\leq s_{1}<s_{2}\leq\arctan k}\frac{\log(1+\tan^{2}s_{2})-\log(1+\tan^{2}s_{1})}{(s_{2}-s_{1})^{\rho}} \\
&\leq\frac{\sqrt{\pi}}{2}\max\bigl(\log(1+k^{2}), \sup_{\substack{s_{2}-s_{1}<1 \\0\leq s_{1}<s_{2}\leq\arctan k}}\frac{\log(1+\tan^{2}s_{2})-\log(1+\tan^{2}s_{1})}{s_{2}-s_{1}}\bigr). \\
\end{split}
\end{equation*}
The function $y=\log(1+\tan^{2}x)$ is increasing and convex on $[0,\infty)$.
So the second term in the last inequality is equal to $(\log(1+\tan^{2}s))'\Bigl |_{s=\arctan k}=2k$.
This completes the proof of (19).}
\end{example}

Now we are in a position to prove Theorem 1.2.
\begin{proof}
It follows from Lemma 3.1 and Theorem 3.6 that
\begin{equation}
|\bigl<n^{-1}F^{\star}_{n}(\omega_{FS}^{k_{n}})-\omega, \psi\bigr>|\leq\frac{C_{1}}{n}\|\psi\|_{\mathscr{C}^{2}}
\end{equation}
for some positive constant $C_{1}$ depending only on $X, L$.
We know that $S_{n}=O(\log n)$ by using Proposition 3.2, (14) and (15),
then Theorem 3.5 and \cite[Lemma 4.2(c)]{ds1} imply that
\begin{equation}
|\bigl<n^{-1}F^{\star}_{n}(\sigma_{n})-n^{-1}F^{\star}_{n}(\omega_{FS}^{k_{n}}), \psi\bigr>|\leq\frac{C_{2}\log n}{n}\|\psi\|_{\mathscr{C}^{2}}
\end{equation}
for some positive constant $C_{2}$ depending only on $X, L$.
Set
\begin{equation*}
E_{n}(\epsilon_{0}):=\bigcup_{\|\psi\|_{\mathscr{C}^{2}}\leq 1}
\{s_{n}\in\mathbb{P}H^{0}(X,L^{n}): |\bigl<n^{-1}[\bm{Z}_{\bm{s_{n}}}]-n^{-1}F^{\star}_{n}(\sigma_{n}), \psi\bigr>|\geq\epsilon_{0}\}
\end{equation*}
for any $\epsilon_{0}>0$.
We define $E_{n}:=E_{n}(\frac{C_{3}\log n}{n})$, where $C_{3}$ is some positive constant depending only on $X, L$.
Note that $R_{n}=O(\log n)$ from inequalities (14) and (15).
By applying \cite[Inequality (4.4)]{ds1}, we deduce that
\begin{equation*}
\sigma_{n}(E_{n})\leq \Delta_{n}(C_{4}\log n).
\end{equation*}
Here $C_{4}$ is a positive constant depending only on $X, L$.
Moreover, $C_{4}$ is sufficiently large such that $\alpha_{0}C_{4}>k+2$
since $C_{3}$ can be chosen sufficiently large. Recall that $\alpha_{0}$ is the constant defined in Proposition 2.2.
Then by (13), we obtain
\begin{equation}
\begin{split}
& \sigma_{n}(E_{n})\leq \Delta_{n}(C_{4}\log n) \\
& \leq c_{0}k_{n}\exp(-\alpha_{0}C_{4}\log n)+c_{5}(\frac{\rho}{4})^{k_{n}}\exp(-\alpha_{0}C_{4}(\frac{\rho}{4})^{k_{n}}\log n) \\
& \leq (c_{0}+c_{5})k_{n}\frac{1}{n^{\alpha_{0}C_{4}}}\leq \frac{C}{n^{2}}. \\
\end{split}
\end{equation}
Here $C$ is a positive constant sufficiently large which depends only on $X, L$.
Note that the third inequality of (22) follows from a direct calculation when $n$ is big enough.
The fact that $k_{n}=O(n^{k})$ yields the last inequality of (22).
By definition of $E_{n}$, we obtain for any point $s_{n}\in\mathbb{P}H^{0}(X,L^{n})\setminus E_{n}$,
\begin{equation}
|\bigl<n^{-1}[\bm{Z}_{\bm{s_{n}}}]-n^{-1}F^{\star}_{n}(\sigma_{n}), \psi\bigr>|
\leq\frac{C_{3}\log n}{n}\|\psi\|_{\mathscr{C}^{2}}.
\end{equation}
It follows from (20),(21) and (23) that
\begin{equation}
|\bigl<n^{-1}[\bm{Z}_{\bm{s_{n}}}]-\omega, \psi\bigr>|\leq\frac{C\log n}{n}\|\psi\|_{\mathscr{C}^{2}}.
\end{equation}
The proof is completed.
\end{proof}

\begin{remark}
Since $\sum_{n=1}^{\infty}\sigma_{n}(E_{n})<\infty$, Theorem 1.2 gives an alternative proof of Theorem 1.1.
\end{remark}

\noindent
G. SHAO,
Universit{\'e} Paris-Sud, Math{\'e}matique - B{\^a}timent 425, 91405
Orsay, France. {\tt shaoguokuang@gmail.com}

\end{document}